\documentclass[reqno,11pt]{article} 
\usepackage[left=2.5cm,right=2.5cm,top=3cm,bottom=3cm,a4paper]{geometry}
\usepackage{amsmath, amssymb}
\usepackage{amsthm, amscd} 
\usepackage[all,cmtip]{xy}
\usepackage{float}
\usepackage{caption}
\usepackage{color}
\usepackage{rotating}
\usepackage[utf8]{inputenc}
\usepackage[T1]{fontenc}

\makeatletter
\DeclareRobustCommand{\rvdots}{%
  \vbox{
    \baselineskip4\p@\lineskiplimit\z@
    \kern-\p@
    \hbox{.}\hbox{.}\hbox{.}
  }}
\makeatother

\newcommand{\Mod}[1]{\ (\textup{mod}\ #1)}
\makeatletter
\def\moverlay{\mathpalette\mov@rlay}
\def\mov@rlay#1#2{\leavevmode\vtop{%
   \baselineskip\z@skip \lineskiplimit-\maxdimen
   \ialign{\hfil$\m@th#1##$\hfil\cr#2\crcr}}}
\newcommand{\charfusion}[3][\mathord]{
    #1{\ifx#1\mathop\vphantom{#2}\fi
        \mathpalette\mov@rlay{#2\cr#3}
      }
    \ifx#1\mathop\expandafter\displaylimits\fi}
\makeatother

\theoremstyle{plain} 
\newtheorem{theorem}{\indent\sc Theorem}[section]
\newtheorem{lemma}[theorem]{\indent\sc Lemma}
\newtheorem{corollary}[theorem]{\indent\sc Corollary}
\newtheorem{proposition}[theorem]{\indent\sc Proposition}

\theoremstyle{definition} 
\newtheorem{definition}[theorem]{\indent\sc Definition}

\newtheorem{remark}[theorem]{\indent\sc Remark}

\makeatletter
\def\address#1#2{\begingroup
\noindent\parbox[t]{7.8cm}{%
\small{\scshape\ignorespaces#1}\par\vskip1ex
\noindent\small{\itshape E-mail address}%
\/: #2\par\vskip4ex}\hfill%
\endgroup}%
\makeatother

\title{Gauss's form class groups and Shimura's canonical models}
\author{
\textsc{Ja Kyung Koo, Dong Hwa Shin, and Dong Sung Yoon$^*$} 
}
\date{} 
\begin{document}

\allowdisplaybreaks

\maketitle

 \footnote{
\thanks{
$^*$Corresponding author.\\
The second named author was supported
by Hankuk University of Foreign Studies Research Fund of 2024 and by the National Research Foundation of Korea (NRF) grant funded by the Korea government (MSIT)
(No. RS-2023-00241953).
}
}

\begin{abstract}
Let $N$ be a positive integer and $\Gamma$ be a subgroup of $\mathrm{SL}_2(\mathbb{Z})$ containing $\Gamma_1(N)$. 
Let $K$ be an imaginary quadratic field and $\mathcal{O}$ be an order of discriminant $D_\mathcal{O}$ in $K$.  
Under some assumptions, we show that 
$\Gamma$ induces a form class group of discriminant $D_\mathcal{O}$ (or of order $\mathcal{O}$) and level $N$
if and only if there is a certain canonical model of the modular curve for $\Gamma$ defined over
a suitably small number field. In this way we can find an interesting link between two different subjects,
which will be useful in the study of certain quadratic Diophantine equations in terms of primes $p$. 
\end{abstract}

\tableofcontents

\section {Introduction}

For a negative integer $D$ such that $D\equiv0$ or $1\Mod{4}$,
let $\mathcal{Q}(D)$ be the set of primitive positive definite binary quadratic forms of discriminant $D$. 
The natural action of $\mathrm{SL}_2(\mathbb{Z})$ on $\mathcal{Q}(D)$ 
defines the proper equivalence $\sim$ on $\mathcal{Q}(D)$, which yields  
the set of equivalence classes $\mathcal{C}(D)=\mathcal{Q}(D)/\sim$. 
Gauss' direct composition or Dirichlet's composition makes 
 $\mathcal{C}(D)$ an abelian group, so-called the \textit{form class group of discriminant $D$}
 (cf. \cite[$\S$3.A]{Cox}). Furthermore, if $\mathcal{O}$ is the order of discriminant $D_\mathcal{O}=D$ 
in the imaginary quadratic field $K=\mathbb{Q}(\sqrt{D})$, then the form class group $\mathcal{C}(D)$ is isomorphic to
the $\mathcal{O}$-ideal class group $\mathcal{C}(\mathcal{O})$ (cf. \cite[$\S$7.B]{Cox}).  
\par
For a positive integer $N$, let $\mathcal{Q}(D,\,N)$ be the subset of $\mathcal{Q}(D)$ consisting of
forms $ax^2+bxy+cy^2$ with $\gcd(a,\,N)=1$. In a similar way to $\mathrm{SL}_2(\mathbb{Z})$, the congruence subgroup
\begin{equation*}
\Gamma_1(N)=\left\{\gamma\in\mathrm{SL}_2(\mathbb{Z})~|~\gamma\equiv\begin{bmatrix}1&\mathrm{*}\\
0&1\end{bmatrix}\Mod{NM_2(\mathbb{Z})}\right\}
\end{equation*} 
acts on $\mathcal{Q}(D,\,N)$ and gives rise to an equivalence relation $\sim_{\Gamma_1(N)}$ on $\mathcal{Q}(D,\,N)$. 
In \cite{J-K-S-Y23} Jung et al. verified that the set of equivalence classes $\mathcal{C}(D,\,N)=\mathcal{Q}(D,\,N)/\sim_{\Gamma_1(N)}$
can be given a binary operation so that it turns out to be an abelian group isomorphic to the ray class group 
$\mathcal{C}(\mathcal{O},\,N)$ of order $\mathcal{O}$ modulo $N\mathcal{O}$, which
makes the diagram in Figure \ref{compatible} commute. 
\begin{figure}[H]
\begin{equation*}
\xymatrixcolsep{10pc}
\xymatrix{
\mathcal{C}(D,\,N) \ar@{->}[r]^{\sim}
\ar@{->>}[dd]_{\textrm{the natural surjection}}
 & \mathcal{C}(\mathcal{O},\,N) \ar@{->}[dd]^{\textrm{the natural homomorphism}} \\\\
\mathcal{C}(D) \ar@{->}[r]^\sim & \mathcal{C}(\mathcal{O})
}
\end{equation*}
\caption{A diagram of class groups}
\label{compatible}
\end{figure}
\par
More generally, let $\Gamma$ be a subgroup of $\mathrm{SL}_2(\mathbb{Z})$ containing $\Gamma_1(N)$. 
Jung et al. also noticed in \cite{J-K-S-Y23-2} that
if there exists a canonical model of the modular curve for $\Gamma$ defined over $\mathbb{Q}$
whose
$\mathbb{Q}$-rational functions are essentially the meromorphic modular functions for $\Gamma$ with 
rational Fourier coefficients, then
$\Gamma$ induces a form class group of discriminant $D_\mathcal{O}$ (or of order $\mathcal{O}$) and level $N$ in the sense of 
Definition \ref{induce}. This assertion can be applied to the subgroup
\begin{equation*}
\Gamma_G=\left\{
\gamma\equiv\begin{bmatrix}t&\mathrm{*}\\0&\mathrm{*}\end{bmatrix}\Mod{NM_2(\mathbb{Z})}~
\textrm{for some $t\in\mathbb{Z}$ such that $t+N\mathbb{Z}\in G$}\right\}
\end{equation*}
of $\mathrm{SL}_2(\mathbb{Z})$, where $G$ is a subgroup of $(\mathbb{Z}/N\mathbb{Z})^\times$ (\cite[Proposition 5.3,
Theorems 3.5 and 5.4]{J-K-S-Y23-2}),
in order to solve the quadratic Diophantine equations $x^2+ny^2=p$ with $x+N\mathbb{Z}\in G$
and $y\equiv0\Mod{N}$ ($n\in\mathbb{Z}$) (\cite[Lemma 6.2 and Theorem 6.4]{J-K-S-Y23-2}).
\par
Now, let $\zeta_N=e^{2\pi\mathrm{i}/N}$. 
In this paper, we shall reveal a close relationship between form class groups and canonical models. 
More precisely, we shall prove that under some mild assumptions
\begin{eqnarray*}
&&\textrm{$\Gamma$ induces a form class group of discriminant $D_\mathcal{O}$ and level $N$}\\
&\Longleftrightarrow&\textrm{there is a canonical model of the modular curve for $\Gamma$ defined over a subfield $L$}\\
&&\textrm{of $\mathbb{Q}(\zeta_N)\cap\mathbb{Q}(\sqrt{D_\mathcal{O}})$ whose $L$-rational functions are essentially}\\
&&\textrm{the meromorphic modular functions for $\Gamma$ with Fourier coefficients in $L$} 
\end{eqnarray*}
(Theorems \ref{S} and \ref{main}). 

\section {Form class groups of level $N$}\label{defineFCG}

We shall give the definition of a form class group of 
discriminant $D_\mathcal{O}$ and level $N$, recently adopted
in \cite{J-K-S-Y23-2}. 
\par
Throughout this paper, we let $K$ be an imaginary quadratic field with 
ring of integers $\mathcal{O}_K$ and 
$\mathcal{O}$ be an order of discriminant $D_\mathcal{O}$ in $K$. 
Let $I(\mathcal{O})$ be the group of proper fractional $\mathcal{O}$-ideals and
$P(\mathcal{O})$ be its subgroup of principal fractional $\mathcal{O}$-ideals.
For a positive integer $N$, we let $I(\mathcal{O},\,N)$ and $P(\mathcal{O},\,N)$ be the
subgroups of $I(\mathcal{O})$ and $P(\mathcal{O})$ defined by 
\begin{align*}
I(\mathcal{O},\,N)&=\langle
\mathfrak{a}~|~\textrm{$\mathfrak{a}$ is a nonzero proper $\mathcal{O}$-ideal prime to $N$}\rangle,\\
P(\mathcal{O},\,N)&=\langle\nu\mathcal{O}~|~\textrm{$\nu$ is a nonzero element of $\mathcal{O}$
such that $\nu\mathcal{O}$ is prime to $N$}\rangle,
\end{align*}
respectively. 
Furthermore, let $P_{1,\,N}(\mathcal{O},\,N)$ be the subgroup of $P(\mathcal{O},\,N)$ given by 
\begin{equation*}
P_{1,\,N}(\mathcal{O},\,N)=\langle\nu\mathcal{O}~|~\textrm{$\nu$ is a nonzero element 
of $\mathcal{O}$ satisfying}~\nu\equiv 1\Mod{N\mathcal{O}}\rangle.
\end{equation*}
We denote its associated quotient group by 
\begin{equation*}
\mathcal{C}(\mathcal{O},\,N)=I(\mathcal{O},\,N)/P_{1,\,N}(\mathcal{O},\,N).
\end{equation*}
Let $\mathcal{Q}(D_\mathcal{O},\,N)$ be the set of binary quadratic forms given by
\begin{equation*}
\mathcal{Q}(D_\mathcal{O},\,N)=\left\{
Q(x,\,y)=Q\left(\begin{bmatrix}x\\y\end{bmatrix}
\right)=ax^2+bxy+cy^2\in\mathbb{Z}[x,\,y]~\Bigg|~\begin{array}{l}\gcd(a,\,b,\,c)=1,\\
b^2-4ac=D_\mathcal{O},~a>0,\\
\gcd(a,\,N)=1\end{array}
\right\}.
\end{equation*}
Each congruence subgroup $\Gamma$ of $\mathrm{SL}_2(\mathbb{Z})$
defines the equivalence relation $\sim_\Gamma$ on the set $\mathcal{Q}(D_\mathcal{O},\,N)$
as follows\,: for $Q,\,Q'\in\mathcal{Q}(D_\mathcal{O},\,N)$
\begin{equation*}
Q\sim_\Gamma Q'\quad\Longleftrightarrow\quad
Q'\left(\begin{bmatrix}x\\y\end{bmatrix}
\right)=
Q\left(\gamma\begin{bmatrix}x\\y\end{bmatrix}\right)~\textrm{for some}~\gamma\in\Gamma. 
\end{equation*}
Then by $\mathcal{C}_\Gamma(D_\mathcal{O},\,N)$ we mean
 the set of equivalence classes, namely, 
\begin{equation*}
\mathcal{C}_\Gamma(D_\mathcal{O},\,N)=\mathcal{Q}(D_\mathcal{O},\,N)/\sim_\Gamma.
\end{equation*}
For any $Q(x,\,y)=ax^2+bxy+cy^2\in\mathcal{Q}(D_\mathcal{O},\,N)$, let $\omega_Q$ be the
zero of the quadratic polynomial $Q(x,\,1)$ which lies in the complex upper half-plane $\mathbb{H}=\{\tau\in\mathbb{C}~|~
\mathrm{Im}(\tau)>0\}$, that is,
\begin{equation*}
\omega_Q=\frac{-b+\sqrt{D_\mathcal{O}}}{2a}. 
\end{equation*}
For the principal form 
\begin{equation*}
Q_\mathcal{O}(x,\,y)=a_\mathcal{O}x^2+b_\mathcal{O}xy+c_\mathcal{O}y^2=
\left\{\begin{array}{ll}
\displaystyle x^2+xy+\frac{1-D_\mathcal{O}}{4}y^2 & \textrm{if}~D_\mathcal{O}\equiv1\Mod{4},\vspace{0.1cm}\\
\displaystyle x^2-\frac{D_\mathcal{O}}{4}y^2 & \textrm{if}~D_\mathcal{O}\equiv0\Mod{4}\\
\end{array}\right.
\end{equation*}
of discriminant $D_\mathcal{O}$, we also write
$\tau_\mathcal{O}$ for $\omega_{Q_\mathcal{O}}$. Note that the lattice
$[\omega_Q,\,1]=\mathbb{Z}\omega_Q+\mathbb{Z}$ in $\mathbb{C}$ belongs
to $I(\mathcal{O},\,N)$ and 
$[\tau_\mathcal{O},\,1]=\mathcal{O}$ (cf. \cite[Lemmas 7.2, 7.5 and (7.16)]{Cox}).

\begin{proposition}\label{formclassgroup}
The set $\mathcal{C}_{\Gamma_1(N)}(D_\mathcal{O},\,N)$ can be equipped with
a unique group structure so that the mapping
\begin{equation*}
\begin{array}{rccc}
\rho_{D_\mathcal{O},\,N}:&\mathcal{C}_{\Gamma_1(N)}(D_\mathcal{O},\,N)&\rightarrow&\mathcal{C}(\mathcal{O},\,N),\\
&[Q]&\mapsto&[[\omega_Q,\,1]]
\end{array}
\end{equation*}
becomes a well-defined isomorphism. 
\end{proposition}
\begin{proof}
See \cite[Definition 5.7 and Theorem 9.4]{J-K-S-Y23}. 
\end{proof}

Let $\ell_\mathcal{O}$ be the conductor of the order $\mathcal{O}$,
and let $P_{1,\,N}(\mathcal{O}_K,\,\ell_\mathcal{O}N)$ be the subgroup of
$I(\mathcal{O}_K,\,\ell_\mathcal{O}N)$ defined by
\begin{equation*}
P_{1,\,N}(\mathcal{O}_K,\,\ell_\mathcal{O}N)
=\left\langle\nu\mathcal{O}_K~\bigg|~
\begin{array}{l}
\textrm{$\nu$ is a nonzero element of $\mathcal{O}_K$ for which
$\nu\mathcal{O}_K$ is prime to $\ell_\mathcal{O}N$}\\
\textrm{and}~\nu\equiv a\Mod{\ell_\mathcal{O}N\mathcal{O}_K}~
\textrm{for some $a\in\mathbb{Z}$ such that $a\equiv1\Mod{N}$}
\end{array}
\right\rangle.
\end{equation*}
Then we have the isomorphism
\begin{equation*}
\begin{array}{ccc}
I(\mathcal{O}_K,\,\ell_\mathcal{O}N)/P_{1,\,N}(\mathcal{O}_K,\,\ell_\mathcal{O}N)
&\stackrel{\sim}{\rightarrow}&\mathcal{C}(\mathcal{O},\,N),\\
\left[\mathfrak{a}\right] &\mapsto&[\mathfrak{a}\cap\mathcal{O}]
\end{array}
\end{equation*}
where $\mathfrak{a}$ stands for a nonzero ideal of $\mathcal{O}_K$ prime to $\ell_\mathcal{O}N$
(\cite[Propositions 2.8 and 2.13]{J-K-S-Y23}). 
On the other hand, 
the existence theorem of
class field theory guarantees that there is a unique abelian extension 
$L$ of $K$ such that 
the generalized ideal class group $I(\mathcal{O}_K,\,\ell_\mathcal{O}N)/P_{1,\,N}(\mathcal{O}_K,\,\ell_\mathcal{O}N)$
is isomorphic to $\mathrm{Gal}(L/K)$ via
the Artin map for the modulus $\ell_\mathcal{O}N\mathcal{O}_K$ (cf. \cite[$\S$V.9]{Janusz}). 
We denote this extension field $L$ of $K$ by $K_{\mathcal{O},\,N}$ and call it the \textit{ray class field of order $\mathcal{O}$ modulo $N\mathcal{O}$},
which plays an important role in the study of solving the quadratic equations
$x^2+ny^2=p$ with additional conditions on $x$ and $y$. We refer to 
\cite{Cho} and \cite{J-K-S-Y23-2} for concrete examples
(or \cite{Cox} without additional conditions on $x$ and $y$). 
\par
Let $\mathcal{F}_N$ be the field of meromorphic modular functions of level $N$ whose
Fourier coefficients belong to the 
$N$th cyclotomic field $\mathbb{Q}(\zeta_N)$. 
Then, $\mathcal{F}_N$ is a Galois extension of $\mathcal{F}_1$ whose Galois group
is isomorphic to $\mathrm{GL}_2(\mathbb{Z}/N\mathbb{Z})/\langle-I_2\rangle$. More precisely,
if $\alpha\in\mathrm{SL}_2(\mathbb{Z}/N\mathbb{Z})/\langle-I_2\rangle$, then for $f\in\mathcal{F}_N$
\begin{equation*}
f^\alpha=f\circ\beta\quad\textrm{where $\beta$ is any element of $\mathrm{SL}_2(\mathbb{Z})$ which reduces to $\alpha$}. 
\end{equation*}
If $\alpha\in\mathrm{GL}_2(\mathbb{Z}/N\mathbb{Z})/\langle-I_2\rangle$ is of the form $\left[\begin{bmatrix}1&0\\0&d\end{bmatrix}\right]$ for some integer $d$ 
relatively prime to $N$, then 
for $\displaystyle f=\sum_{n\gg-\infty}c_nq_\tau^n\in\mathcal{F}_N$ ($c_n\in\mathbb{Q}(\zeta_N)$, $q_\tau=e^{2\pi\mathrm{i}\tau}$)
\begin{equation*}
f^\alpha=\sum_{n\gg-\infty}c_n^{\sigma_d}q_\tau^n
\quad\textrm{where $\sigma_d$ is the element of $\mathrm{Gal}(\mathbb{Q}(\zeta_N)/\mathbb{Q})$
defined by $\zeta_N\mapsto\zeta_N^d$}
\end{equation*}
(cf. \cite[Theorem 3 in Chapter 6]{Lang} and 
\cite[Proposition 6.9 (1)]{Shimura}). By the theory of complex multiplication, we get that
\begin{equation}\label{generation1}
K_{\mathcal{O},\,N}=K\left(
f(\tau_\mathcal{O})~|~f\in\mathcal{F}_N~\textrm{is finite at $\tau_\mathcal{O}$}\right)
\end{equation}
(\cite[Theorem 4]{Cho}). 
If we let $\mathcal{F}_{\Gamma_1(N),\,\mathbb{Q}}$ be the
field of meromorphic modular functions for $\Gamma_1(N)$ with rational Fourier coefficients,
then we also have
\begin{equation}\label{generation2}
K_{\mathcal{O},\,N}=K\left(
f(\tau_\mathcal{O})~|~f\in\mathcal{F}_{\Gamma_1(N),\,\mathbb{Q}}~\textrm{is finite at $\tau_\mathcal{O}$}\right)
\end{equation}
(\cite[Theorem 5]{Cho}).
And, Jung-Koo-Shin-Yoon recently established an explicit
isomorphism between $\mathcal{C}_{\Gamma_1(N)}(D_\mathcal{O},\,N)$ 
and $\mathrm{Gal}(K_{\mathcal{O},\,N}/K)$. 

\begin{proposition}\label{Galois}
The map
\begin{eqnarray*}
\phi_{D_\mathcal{O},\,N}~:~\mathcal{C}_{\Gamma_1(N)}(D_\mathcal{O},\,N) & \rightarrow & \mathrm{Gal}(K_{\mathcal{O},\,N}/K)\\
\mathrm{[}Q]=[ax^2+bxy+cy^2] & \mapsto & \left(
f(\tau_\mathcal{O})\mapsto f^{\left[
\begin{smallmatrix}1 & -\check{a}\left(\frac{b+b_\mathcal{O}}{2}\right)\\0&\check{a}
\end{smallmatrix}\right]}(-\overline{\omega}_Q)~\bigg|~f\in\mathcal{F}_N~\textrm{is finite at $\tau_\mathcal{O}$}\right)
\end{eqnarray*}
is a well-defined isomorphism, where $\check{a}$ is any integer satisfying $a\check{a}\equiv1\Mod{N}$. 
\end{proposition}
\begin{proof}
See \cite[Theorem 12.3]{J-K-S-Y23}. 
\end{proof}

\begin{remark}\label{cyclotomic}
\begin{enumerate}
\item[(i)]
By composing the isomorphism $\phi_{D_\mathcal{O},\,N}$ 
given in Proposition \ref{Galois}
and the restriction map
$\mathrm{Gal}(K_{\mathcal{O},\,N}/K)\rightarrow
\mathrm{Gal}(K(\zeta_N)/K)$, we obtain the surjective homomorphism
\begin{equation*}
\begin{array}{ccc}
\mathcal{C}_{\Gamma_1(N)}(D_\mathcal{O},\,N)&\twoheadrightarrow&\mathrm{Gal}(K(\zeta_N)/K),\\ 
\left[ax^2+bxy+cy^2\right]&\mapsto&\left(\zeta_N\mapsto\zeta_N^{\check{a}}\right).
\end{array}
\end{equation*}
Thus, $\{a+N\mathbb{Z}~|~ax^2+bxy+cy^2\in\mathcal{Q}(D_\mathcal{O},\,N)\}$ is
a subgroup of $(\mathbb{Z}/N\mathbb{Z})^\times$ ($\simeq\mathrm{Gal}(\mathbb{Q}(\zeta_N)/\mathbb{Q})$) 
of index 
\begin{equation*}
\frac{[\mathbb{Q}(\zeta_N):\mathbb{Q}]}{[K(\zeta_N):K]}=\frac{[K:\mathbb{Q}]}{[K(\zeta_N):\mathbb{Q}(\zeta_N)]}=
\left\{\begin{array}{ll}
1 & \textrm{if}~K\not\subseteq\mathbb{Q}(\zeta_N),\\
2 & \textrm{if}~K\subseteq\mathbb{Q}(\zeta_N).
\end{array}\right.
\end{equation*}
Moreover, if we let $k_{D_\mathcal{O},\,N}$ be the fixed field of $\mathbb{Q}(\zeta_N)$
by the subgroup $\{a+N\mathbb{Z}~|~ax^2+bxy+cy^2\in\mathcal{Q}(D_\mathcal{O},\,N)\}$
of $(\mathbb{Z}/N\mathbb{Z})^\times$, 
then we attain
\begin{equation*}
k_{D_\mathcal{O},\,N}=\mathbb{Q}(\zeta_N)\cap K=\mathbb{Q}~\textrm{or}~K. 
\end{equation*}
\item[(ii)] 
Observe that $K(\zeta_N)$ is Galois over $\mathbb{Q}$ and
the map $\mathfrak{c}\in\mathrm{Gal}(K(\zeta_N)/\mathbb{Q})$ obtained by restricting the complex conjugation on $\mathbb{C}$
does not belong to $\mathrm{Gal}(K(\zeta_N)/K)$.  
If $k_{D_\mathcal{O},\,N}=K$, that is, if $K\subseteq\mathbb{Q}(\zeta_N)$,  
then we see that
$\mathfrak{c}$ ($\not\in\mathrm{Gal}(K(\zeta_N)/K)$)  is the only element of $\mathrm{Gal}(K(\zeta_N)/\mathbb{Q})=
\mathrm{Gal}(\mathbb{Q}(\zeta_N)/\mathbb{Q})$ which
maps $\zeta_N$ to $\zeta_N^{-1}$. 
\par
We claim that if $-1$ is a quadratic residue modulo $N$, then $k_{D_\mathcal{O},\,N}=\mathbb{Q}$. Indeed, let $t$ be a nonzero integer such that $t^2\equiv-1\Mod{N}$,
and let $p$ be a prime such that $p\equiv t\Mod{N}$. The action of the
Artin map $\left(\frac{K(\zeta_N)/K}{p\mathcal{O}_K}\right)$ ($\in\mathrm{Gal}(K(\zeta_N)/K)$) on $\zeta_N$ is given by 
\begin{eqnarray*}
\zeta_N^{\left(\frac{K(\zeta_N)/K}{p\mathcal{O}_K}\right)}=
\zeta_N^{\left(\frac{\mathbb{Q}(\zeta_N)/\mathbb{Q}}{N_{K/\mathbb{Q}}(p\mathcal{O}_K)}\right)}=
\zeta_N^{\left(\frac{\mathbb{Q}(\zeta_N)/\mathbb{Q}}{p^2\mathbb{Z}}\right)}=
\zeta_N^{t^2}=\zeta_N^{-1} 
\end{eqnarray*}
(cf. \cite[Proposition 3.1 in Chapter III]{Janusz}). This observation together with the preceding paragraph implies that
$k_{D_\mathcal{O},\,N}=\mathbb{Q}(\zeta_N)\cap K=\mathbb{Q}$. 
\end{enumerate}
\end{remark}

\begin{definition}{(\cite[Definition 4.3]{J-K-S-Y23-2})}\label{induce}
Let $N$ be a positive integer. 
We say that a subgroup $\Gamma$ of $\mathrm{SL}_2(\mathbb{Z})$ 
\textit{induces a form class group of discriminant $D_\mathcal{O}$ (or of order $\mathcal{O}$) and level $N$} if
\begin{enumerate}
\item[(i)] $\Gamma$ contains $\Gamma_1(N)$,
\item[(ii)] there is a subgroup $P$ of $P(\mathcal{O},\,N)$ containing $P_{1,\,N}(\mathcal{O},\,N)$ for which the mapping
\begin{equation*}
\begin{array}{ccc}
\mathcal{C}_\Gamma(D_\mathcal{O},\,N)&\rightarrow& I(\mathcal{O},\,N)/P,\\
\left[Q\right]&\mapsto&[[\omega_Q,\,1]]
\end{array}
\end{equation*}
is a well-defined bijection. 
\end{enumerate}
In this case, we regard $\mathcal{C}_\Gamma(D_\mathcal{O},\,N)$ as a group isomorphic to $I(\mathcal{O},\,N)/P$,
and call it a \textit{form class group of discriminant $D_\mathcal{O}$ and level $N$}. 
\end{definition}

\begin{remark}\label{full}
\begin{enumerate}
\item[(i)] By Proposition \ref{formclassgroup}, $\Gamma_1(N)$ induces a form class group of
discriminant $D_\mathcal{O}$ and level $N$. 
\item[(ii)] The full modular group $\mathrm{SL}_2(\mathbb{Z})$ also induces
a form class group of discriminant $D_\mathcal{O}$ and level $N$ because the map
\begin{equation*}
\begin{array}{ccc}
\mathcal{C}_{\mathrm{SL}_2(\mathbb{Z})}(D_\mathcal{O},\,N)&
\rightarrow& I(\mathcal{O},\,N)/P(\mathcal{O},\,N),\\
\left[Q\right]&\mapsto&[[\omega_Q,\,1]]
\end{array}
\end{equation*}
is a well-defined bijection (\cite[Corollary 7.17]{Cox} and \cite[Lemmas 2.2 and 2.11]{J-K-S-Y23}). 
\end{enumerate}
\end{remark}

\begin{remark}
If a subgroup $\Gamma$ of $\mathrm{SL}_2(\mathbb{Z})$ induces a form class group of discriminant $D_\mathcal{O}$ and level $N$, 
then we see from the commutative diagram in Figure \ref{homo}
that the natural surjection $\mathcal{C}_{\Gamma_1(N)}(D_\mathcal{O},\,N)
\twoheadrightarrow\mathcal{C}_\Gamma(D_\mathcal{O},\,N)$
becomes a homomorphism.  
\begin{figure}[H]
\begin{equation*}
\xymatrixcolsep{10pc}
\xymatrix{
\mathcal{C}_{\Gamma_1(N)}(D_\mathcal{O},\,N) \ar@{->}[r]^{\sim}
\ar@{->>}[dd]_{\textrm{the natural surjection}}
 & \mathcal{C}(\mathcal{O},\,N) \ar@{->>}[dd]^{\textrm{the natural homomorphism}} \\\\
\mathcal{C}_\Gamma(D_\mathcal{O},\,N) \ar@{->}[r]^\sim & I(\mathcal{O},\,N)/P
}
\end{equation*}
\caption{A commutative diagram of class groups}
\label{homo}
\end{figure}

Therefore, the operation of $\mathcal{C}_\Gamma(D_\mathcal{O},\,N)$ is directly induced from that of 
$\mathcal{C}_{\Gamma_1(N)}(D_\mathcal{O},\,N)$. 
\end{remark}

\section {Canonical models of modular curves}\label{canonicalmodel}

In this section, we shall briefly recall the theory of canonical models
of modular curves established by Shimura. 
\par
Let
\begin{equation*}
\mathcal{F}=\bigcup_{N=1}^\infty\mathcal{F}_N,\quad
\widehat{\mathbb{Z}}=\prod_{p\,:\,\textrm{primes}}\mathbb{Z}_p
\quad\textrm{and}\quad
\widehat{\mathbb{Q}}=\mathbb{Q}\otimes_\mathbb{Z}\widehat{\mathbb{Z}}. 
\end{equation*}
Then we have an exact sequence
\begin{equation*}
1\rightarrow\mathbb{Q}^\times\rightarrow\mathrm{GL}_2(\widehat{\mathbb{Q}})
\rightarrow\mathrm{Gal}(\mathcal{F}/\mathbb{Q})\rightarrow1
\end{equation*}
which can be described as follows\,: let $N$ be a positive integer and $\gamma\in\mathrm{GL}_2(\widehat{\mathbb{Q}})$.
One can decompose $\gamma$ as
\begin{equation*}
\gamma=\alpha\beta\quad\textrm{for some}~\alpha=(\alpha_p)_p\in\mathrm{GL}_2(\widehat{\mathbb{Z}})
=\prod_{p\,:\,\textrm{primes}}\mathrm{GL}_2(\mathbb{Z}_p)~
\textrm{and}~\beta\in\mathrm{GL}_2^+(\mathbb{Q}). 
\end{equation*}
Furthermore, by using the Chinese remainder theorem
one can take a matrix $\alpha'$ in $M_2(\mathbb{Z})$
such that $\alpha'\equiv\alpha_p\Mod{NM_2(\mathbb{Z}_p)}$ for all primes $p$.
If we let $\alpha''$ be the reduction of $\alpha'$ to $\mathrm{GL}_2(\mathbb{Z}/N\mathbb{Z})/\langle-I_2\rangle$
($\simeq\mathrm{Gal}(\mathcal{F}_N/\mathcal{F}_1)$), then 
the action of $\gamma$ on $\mathcal{F}_N$ is given by
\begin{equation*}
f^\gamma=f^{\alpha''}\circ\beta\quad(f\in\mathcal{F}_N)
\end{equation*}
(\cite[Theorem 6.23]{Shimura}).
\par
Let $S$ be an open subgroup of $\mathrm{GL}_2(\widehat{\mathbb{Q}})$
containing $\mathbb{Q}^\times$ such that $S/\mathbb{Q}^\times$ is compact. We define
\begin{align*}
\Gamma_S&=S\cap\mathrm{GL}_2^+(\mathbb{Q}),\\
\mathcal{F}_S&=\left\{f\in\mathcal{F}~|~f^\gamma=f~\textrm{for all}~\gamma\in S\right\},\\
k_S&=\left\{
\nu\in\mathbb{Q}^\mathrm{ab}~|~\nu^{[s^{-1},\,\mathbb{Q}]}=\nu~\textrm{for all}~
s\in\det(S)\right\}.
\end{align*}
Here, $[\,\cdot,\,\mathbb{Q}]:\widehat{\mathbb{Q}}^\times\rightarrow
\mathrm{Gal}(\mathbb{Q}^\mathrm{ab}/\mathbb{Q})$ means the Artin map for $\mathbb{Q}$,
where $\mathbb{Q}^\mathrm{ab}$ is the maximal abelian extension of $\mathbb{Q}$. 

\begin{proposition}\label{FSkS}
With the notations as above, we get that
\begin{enumerate}
\item[\textup{(i)}] $\Gamma_S/\mathbb{Q}^\times$ is a Fuchsian group of the first kind 
commensurable with $\mathrm{SL}_2(\mathbb{Z})/\langle-I_2\rangle$.
\item[\textup{(ii)}] $\mathbb{C}\mathcal{F}_S$ is the field of 
meromorphic modular functions for $\Gamma_S/\mathbb{Q}^\times$.
\item[\textup{(iii)}] $k_S$ is algebraically closed in $\mathcal{F}_S$. 
\end{enumerate}
\end{proposition}
\begin{proof}
See \cite[Proposition 6.27]{Shimura}. 
\end{proof}

Let $\mathbb{H}^*=\mathbb{H}\cup\mathbb{Q}\cup\{\mathrm{i}\infty\}$. 
By virtue of Proposition \ref{FSkS}, we can consider a \textit{canonical model}
$(V_S,\,\varphi_S)$ of the modular curve $\Gamma_S\backslash\mathbb{H}^*$,
where $V_S$ is a nonsingular projective curve defined over $\mathbb{C}$ and 
$\varphi_S:\mathbb{H}^*\rightarrow
V_S$ is a $\Gamma_S$-invariant holomorphic map which gives a biregular isomorphism of $\Gamma_S
\backslash\mathbb{H}^*$ onto $V_S$, 
characterized by the following properties\,:
\begin{enumerate}
\item[(i)] $V_S$ is defined over $k_S$
\item[(ii)] $\mathcal{F}_S=\{h\circ\varphi_S~|~\textrm{$h$ is a $k_S$-rational function 
on $V_S$ in the sense of algebraic geometry}\}$. 
\end{enumerate} 

\begin{lemma}\label{twopoints}
Let $\omega$ and $\omega'$ be elements of $\mathbb{H}$ such that
\begin{equation*}
\{f\in\mathcal{F}_S~|~\textrm{$f$ is finite at $\omega$}\}=
\{f\in\mathcal{F}_S~|~\textrm{$f$ is finite at $\omega'$}\}. 
\end{equation*}
If $f(\omega)=f(\omega')$ for all $f\in\mathcal{F}_S$ finite at $\omega$, then
$\omega=\gamma(\omega')$ for some $\gamma\in\Gamma_S/\mathbb{Q}^\times$. 
\end{lemma}
\begin{proof}
See \cite[Lemma 7.1]{J-K-S-Y223}. 
\end{proof}

 Let
 \begin{equation*}
\Delta_S=\left\{\begin{bmatrix}
1&0\\0&s
\end{bmatrix}~|~s\in\widehat{\mathbb{Z}}^\times~\textrm{satisfies that $[s^{-1},\,\mathbb{Q}]$
is the identity on $k_S$}\right\}.
\end{equation*}
 
\begin{lemma}\label{DeltaLemma}
If $S$ contains $\Delta_S$, then $\mathcal{F}_S$ is exactly the field of meromorphic
modular functions for $\Gamma_S/\mathbb{Q}^\times$ whose Fourier coefficients belong to $k_S$. 
\end{lemma}
\begin{proof}
See \cite[Theorem 6]{Shimura75}.
\end{proof}

Let $\widehat{K}=K\otimes_\mathbb{Z}\widehat{\mathbb{Z}}$
and $K^\mathrm{ab}$ be the maximal abelian extension of $K$. 
Then it follows from the class field theory that 
the Artin map $\widehat{K}^\times\rightarrow\mathrm{Gal}(K^\mathrm{ab}/K)$ yields a one-to-one
correspondence 
\begin{eqnarray*}
\{\textrm{closed subgroups of $\widehat{K}^\times$ of finite index containing $K^\times$}\}
&\rightarrow&\{\textrm{finite abelian extensions of $K$}\}\\
J&\mapsto&\textrm{$L$ satisfying $\widehat{K}^\times/J\simeq\mathrm{Gal}(L/K)$}
\end{eqnarray*}
(cf. \cite[$\S$IV.7]{Neukirch}). 
Let $\omega\in K\cap\mathbb{H}$. By continuity, the embedding
$q_\omega:K^\times\rightarrow\mathrm{GL}_2^+(\mathbb{Q})$ defined by 
\begin{equation*}
\tau\begin{bmatrix}\omega\\1\end{bmatrix}
=q_\omega(\tau)\begin{bmatrix}\omega\\1\end{bmatrix}\quad(\tau\in K^\times)
\end{equation*}
can be extended to an embedding $(K\otimes_\mathbb{Z}\mathbb{Z}_p)^\times\rightarrow\mathrm{GL}_2(\mathbb{Q}_p)$ for each prime $p$.
Thus we attain an embedding $\widehat{q}_\omega:\widehat{K}^\times\rightarrow\mathrm{GL}_2(\widehat{\mathbb{Q}})$.

\begin{proposition}
With the notations as above, the subgroup $K^\times\widehat{q}_\omega^{\,-1}(S)$ of $\widehat{K}^\times$ corresponds to the subfield $K\left(f(\omega)~|~f\in\mathcal{F}_S~\textrm{is finite at}~\omega\right)$ of $K^\mathrm{ab}$. 
\end{proposition}
\begin{proof}
See \cite[Proposition 6.33]{Shimura}.
\end{proof}

\section {Form class groups obtained by canonical models}

We shall show that 
a subgroup $\Gamma$ of $\mathrm{SL}_2(\mathbb{Z})$, which satisfies that
$\langle\Gamma,\,-I_2\rangle/\langle-I_2\rangle\simeq\Gamma_S/\mathbb{Q}^\times$
for a certain subgroup $S$ of $\mathrm{GL}_2(\widehat{\mathbb{Q}})$, 
induces a form class group stated in Definition \ref{induce}.
\par
Let $k_{D_\mathcal{O},\,N}$ be the fixed field of $\mathbb{Q}(\zeta_N)$ by the subgroup
$\{a+N\mathbb{Z}~|~Q=ax^2+bxy+cy^2\in\mathcal{Q}(D_\mathcal{O},\,N)\}$
of $(\mathbb{Z}/N\mathbb{Z})^\times$. 
By Remark \ref{cyclotomic} (i), we have $k_{D_\mathcal{O},\,N}=\mathbb{Q}(\zeta_N)\cap K$ which is $\mathbb{Q}$ or $K$. 

\begin{lemma}\label{bar}
Let $f$ be a meromorphic modular function for $\Gamma_1(N)$
whose Fourier coefficients belong to $\mathbb{Q}(\zeta_N)$, and let $z\in\mathbb{H}$.  
If $f$ is finite at $z$, then 
$f^{\left[\begin{smallmatrix}1&0\\0&-1\end{smallmatrix}\right]}$ is finite at $-\overline{z}$ and 
\begin{equation*}
f^{\left[\begin{smallmatrix}1&0\\0&-1\end{smallmatrix}\right]}(-\overline{z})=\overline{f(z)}. 
\end{equation*}
\end{lemma}
\begin{proof}
If we let $\displaystyle f(\tau)=\sum_{n\gg-\infty}c_nq_\tau^n$ ($c_n\in\mathbb{C}$, $\tau\in\mathbb{H}$)
be the Fourier expansion of $f$ with respect to $q_\tau=e^{2\pi\mathrm{i}\tau}$, then we find that
\begin{equation*}
\overline{f(z)}=\overline{\sum_{n\gg-\infty}c_ne^{2\pi\mathrm{i}nz}}=
\sum_{n\gg-\infty}\overline{c_n}e^{2\pi\mathrm{i}n(-\overline{z})}=
f^{\left[\begin{smallmatrix}1&0\\0&-1\end{smallmatrix}\right]}(-\overline{z}).
\end{equation*}
\end{proof}

\begin{theorem}\label{S}
Let $N$ be a positive integer and $\Gamma$ be a subgroup of $\mathrm{SL}_2(\mathbb{Z})$ containing $\Gamma_1(N)$. 
Let $K$ be an imaginary quadratic field and $\mathcal{O}$ be an order of discriminant $D_\mathcal{O}$ in $K$. 
If there is an open subgroup $S$ of $\mathrm{GL}_2(\widehat{\mathbb{Q}})$ containing $\mathbb{Q}^\times$
such that $S/\mathbb{Q}^\times$ is compact and
\begin{enumerate}
\item[\textup{(i)}] $\Gamma_S/\mathbb{Q}^\times\simeq\langle\Gamma,\,-I_2\rangle/\langle-I_2\rangle$,
\item[\textup{(ii)}] $k_S$ is a subfield of $k_{D_\mathcal{O},\,N}=\mathbb{Q}(\zeta_N)\cap K$, 
\item[\textup{(iii)}] $S$ contains $\Delta_S$,
\end{enumerate}
then $\Gamma$ induces a form class group of discriminant $D_\mathcal{O}$ and level $N$. 
\end{theorem}
\begin{proof}
Let $F$ be the field of meromorphic modular functions for $\Gamma_1(N)$ whose
Fourier coefficients belong to $k_{D_\mathcal{O},\,N}$. 
We observe that if $f\in F$ is finite at $\tau_\mathcal{O}$, then 
\begin{eqnarray*}
\overline{f(\tau_\mathcal{O})}&=&f^{\left[\begin{smallmatrix}1&0\\0&-1\end{smallmatrix}\right]}(-\overline{\tau}_\mathcal{O})
\quad\textrm{by Lemma \ref{bar}}\\
&=&f^{\left[\begin{smallmatrix}1&0\\0&-1\end{smallmatrix}\right]}\left(\begin{bmatrix}1&b_\mathcal{O}\\0&1\end{bmatrix}(\tau_\mathcal{O})\right)\\
&=&f^{\left[\begin{smallmatrix}1&0\\0&-1\end{smallmatrix}\right]
\left[\begin{smallmatrix}1&b_\mathcal{O}\\0&1\end{smallmatrix}\right]}(\tau_\mathcal{O})
\quad\textrm{since $f^{\left[\begin{smallmatrix}1&0\\0&-1\end{smallmatrix}\right]}$ belongs to $\mathcal{F}_N$}\\
&=&f^{\left[\begin{smallmatrix}1&b_\mathcal{O}\\0&-1\end{smallmatrix}\right]}(\tau_\mathcal{O})\\
&=&f^{\left[\begin{smallmatrix}1&-b_\mathcal{O}\\0&1\end{smallmatrix}\right]
\left[\begin{smallmatrix}1&0\\0&-1\end{smallmatrix}\right]}(\tau_\mathcal{O})\\
&=&f^{\left[\begin{smallmatrix}1&0\\0&-1\end{smallmatrix}\right]}(\tau_\mathcal{O})\quad\textrm{because $f$ is modular for $\Gamma_1(N)$}.
\end{eqnarray*}
Moreover, since 
\begin{equation*}
\mathcal{F}_{\Gamma_1(N),\,\mathbb{Q}}\subseteq\left\{f^{\left[\begin{smallmatrix}1&0\\0&-1\end{smallmatrix}\right]}~|~
f\in F\right\}\subseteq\mathcal{F}_N, 
\end{equation*}
we get by (\ref{generation1}) and (\ref{generation2}) that
\begin{equation}\label{KKF}
K_{\mathcal{O},\,N}=K\left(f^{\left[\begin{smallmatrix}1&0\\0&-1\end{smallmatrix}\right]}(\tau_\mathcal{O})~|~f\in F~\textrm{is finite at $\tau_\mathcal{O}$}\right). 
\end{equation}
Let $Q=ax^2+bxy+cy^2\in\mathcal{Q}(D_\mathcal{O},\,N)$ and $f\in F$ be finite at $\tau_\mathcal{O}$. 
If we let $\phi=\phi_{D_\mathcal{O},\,N}:
\mathcal{C}_{\Gamma_1(N)}(D_\mathcal{O},\,N)\stackrel{\sim}{\rightarrow}
\mathrm{Gal}(K_{\mathcal{O},\,N}/K)$ be the isomorphism stated in Proposition \ref{Galois}, then we achieve that
\begin{eqnarray*}
f^{\left[\begin{smallmatrix}1&0\\0&-1\end{smallmatrix}\right]}(\tau_\mathcal{O})^{\phi([Q])}&=&
f^{\left[\begin{smallmatrix}1&0\\0&-1\end{smallmatrix}\right]\left[
\begin{smallmatrix}1 & -\check{a}\left(\frac{b+b_\mathcal{O}}{2}\right)\\0&\check{a}
\end{smallmatrix}\right]}(-\overline{\omega}_Q)\quad\textrm{where $\check{a}$
is an integer such that $a\check{a}\equiv1\Mod{N}$}\\
&&\hspace{5cm}\textrm{by Proposition \ref{Galois}}\\
&=&f^{\left[\begin{smallmatrix}1 & -\check{a}\left(\frac{b+b_\mathcal{O}}{2}\right)\\0&-\check{a}
\end{smallmatrix}\right]}(-\overline{\omega}_Q)\\
&=&f^{\left[\begin{smallmatrix}1 & \frac{b+b_\mathcal{O}}{2}\\0&1
\end{smallmatrix}\right]
\left[\begin{smallmatrix}1 & 0\\0&\check{a}
\end{smallmatrix}\right]\left[\begin{smallmatrix}1&0\\0&-1\end{smallmatrix}\right]}(-\overline{\omega}_Q)\\
&=&f^{\left[\begin{smallmatrix}1 & 0\\0&\check{a}
\end{smallmatrix}\right]\left[\begin{smallmatrix}1&0\\0&-1\end{smallmatrix}\right]}(-\overline{\omega}_Q)\quad\textrm{since $f$ is modular for 
$\Gamma_1(N)$}\\
&=&f^{\left[\begin{smallmatrix}1&0\\0&-1\end{smallmatrix}\right]}(-\overline{\omega}_Q)\quad\textrm{because
the Fourier coefficients of $f$ lie in $k_{D_\mathcal{O},\,N}$}. 
\end{eqnarray*}
Thus $\phi$ can be described as
\begin{eqnarray*}
\phi~:~\mathcal{C}_{\Gamma_1(N)}(D_\mathcal{O},\,N) & \stackrel{\sim}{\rightarrow} & \mathrm{Gal}(K_{\mathcal{O},\,N}/K)\\
\mathrm{[}Q]=[ax^2+bxy+cy^2] & \mapsto & \left(
f^{\left[\begin{smallmatrix}1&0\\0&-1\end{smallmatrix}\right]}(\tau_\mathcal{O})\mapsto f^{\left[\begin{smallmatrix}1&0\\0&-1\end{smallmatrix}\right]}(-\overline{\omega}_Q)~|~f\in F~\textrm{is finite at $\tau_\mathcal{O}$}\right).\end{eqnarray*}
By (i), (iii) and
Lemma \ref{DeltaLemma}, $\mathcal{F}_S$ is equal to the field of meromorphic modular functions
for $\Gamma$ ($\supseteq\Gamma_1(N)$) whose Fourier coefficients belong to $k_S$ 
($\subseteq k_{D_\mathcal{O},\,N}$). 
So we deduce from (\ref{KKF}) that the field
\begin{equation*}
L=K\left(
f^{\left[\begin{smallmatrix}1&0\\0&-1\end{smallmatrix}\right]}(\tau_\mathcal{O})~|~f\in\mathcal{F}_S~\textrm{is finite at}~\tau_\mathcal{O}\right)
\end{equation*}
is a subfield of $K_{\mathcal{O},\,N}$, from which we obtain the surjective homomorphism
\begin{eqnarray*}
\phi'~:~\mathcal{C}_{\Gamma_1(N)}(D_\mathcal{O},\,N) &\twoheadrightarrow&\mathrm{Gal}(L/K)\\
\mathrm{[}Q] & \mapsto & \left(
f^{\left[\begin{smallmatrix}1&0\\0&-1\end{smallmatrix}\right]}(\tau_\mathcal{O})\mapsto
f^{\left[\begin{smallmatrix}1&0\\0&-1\end{smallmatrix}\right]}(-\overline{\omega}_Q)~|~f\in\mathcal{F}_S~\textrm{is finite at}~\tau_\mathcal{O}\right).
\end{eqnarray*}
\par
Now, we want to verify that the map
\begin{eqnarray*}
\psi~:~\mathcal{C}_\Gamma(D_\mathcal{O},\,N) &\rightarrow&\mathrm{Gal}(L/K)\\
\mathrm{[}Q] & \mapsto & \left(
f^{\left[\begin{smallmatrix}1&0\\0&-1\end{smallmatrix}\right]}(\tau_\mathcal{O})\mapsto
f^{\left[\begin{smallmatrix}1&0\\0&-1\end{smallmatrix}\right]}(-\overline{\omega}_Q)~|~f\in\mathcal{F}_S~\textrm{is finite at}~\tau_\mathcal{O}\right)
\end{eqnarray*}
is a well-defined bijection. We deduce that for $Q,\,Q'\in\mathcal{Q}(D_\mathcal{O},\,N)$
\begin{eqnarray*}
&&f^{\left[\begin{smallmatrix}1&0\\0&-1\end{smallmatrix}\right]}(-\overline{\omega}_Q)=
f^{\left[\begin{smallmatrix}1&0\\0&-1\end{smallmatrix}\right]}(-\overline{\omega}_{Q'})~
\textrm{for all $f\in\mathcal{F}_S$ which are finite at $\tau_\mathcal{O}$}\\
&\Longleftrightarrow&\overline{f^{\left[\begin{smallmatrix}1&0\\0&-1\end{smallmatrix}\right]}(-\overline{\omega}_Q)}=
\overline{f^{\left[\begin{smallmatrix}1&0\\0&-1\end{smallmatrix}\right]}(-\overline{\omega}_{Q'})}~\textrm{for all $f\in\mathcal{F}_S$ which are finite at $\tau_\mathcal{O}$}\\
&\Longleftrightarrow&f(\omega_Q)=f(\omega_{Q'})
~\textrm{for all $f\in\mathcal{F}_S$ which are finite at $\tau_\mathcal{O}$ by (ii)
and Lemma \ref{bar}}\\
&\Longleftrightarrow&\omega_Q=\gamma(\omega_{Q'})~\textrm{for some $\gamma\in\Gamma$ by Lemma \ref{twopoints}}\\
&\Longleftrightarrow&[Q]=[Q']~\textrm{in}~\mathcal{C}_\Gamma(D_\mathcal{O},\,N). 
\end{eqnarray*}
This argument implies that $\psi$ is a well-defined injection. Moreover, 
we see from the commutative diagram in Figure \ref{comm1} that $\psi$ is surjective,
hence $\psi$ is a well-defined bijection as desired.
\begin{figure}[H]
\begin{equation*}
\xymatrixcolsep{3pc}
\xymatrix{
\mathcal{C}_{\Gamma_1(N)}(D_\mathcal{O},\,N) \ar@{->>}[rr]^{\textrm{the natural surjection}}^{}
\ar@{->>}[ddr]_{\phi'}^{\rotatebox[origin=c]{150}{}}
 & & \mathcal{C}_{\Gamma}(D_\mathcal{O},\,N) \ar@{->}[ddl]^{\psi}\\\\
& \mathrm{Gal}(L/K)  & }
\end{equation*}
\caption{A commutative diagram for surjectivity of $\psi$}
\label{comm1}
\end{figure}
\par
Let $\rho=\rho_{D_\mathcal{O},\,N}:\mathcal{C}_{\Gamma_1(N)}(D_\mathcal{O},\,N)
\stackrel{\sim}{\rightarrow}I(\mathcal{O},\,N)/P_{1,\,N}(\mathcal{O},\,N)$ be the isomorphism given in Proposition 
\ref{formclassgroup}, and let $P$ be the subgroup of $I(\mathcal{O},\,N)$ containing
$P_{1,\,N}(\mathcal{O},\,N)$ so that the 
image of the subgroup $\mathrm{Gal}(K_{\mathcal{O},\,N}/L)$ of 
$\mathrm{Gal}(K_{\mathcal{O},\,N}/K)$ under
the isomorphism
\begin{equation*}
\begin{array}{rcccl}
\rho\circ\phi^{-1}:&\mathrm{Gal}(K_{\mathcal{O},\,N}/K)& \stackrel{\sim}{\rightarrow}&
I(\mathcal{O},\,N)/P_{1,\,N}(\mathcal{O},\,N),&\\
&\phi([Q])&\mapsto&[[\omega_Q,\,1]]&\textrm{for $Q\in\mathcal{Q}(D_\mathcal{O},\,N)$}
\end{array}
\end{equation*}
is $P/P_{1,\,N}(\mathcal{O},\,N)$. Then we get the commutative diagram in Figure \ref{Figure4}, 
which yields the well-defined bijection
\begin{equation*}
\begin{array}{ccc}
\mathcal{C}_\Gamma(D_\mathcal{O},\,N)&\rightarrow& I(\mathcal{O},\,N)/P,\\
 \left[Q\right]& \mapsto &[[\omega_Q,\,1]].
 \end{array}
\end{equation*}
Furthermore, we note by Remark \ref{full} (ii) that $P$ is a subgroup of $P(\mathcal{O},\,N)$.
Therefore $\Gamma$ induces a form class group of discriminant $D_\mathcal{O}$ and level $N$. 
\begin{figure}[H]
\begin{equation*}
\xymatrixcolsep{7pc}
\xymatrix{
\mathcal{C}_{\Gamma_1(N)}(D_\mathcal{O},\,N) \ar@{->}[r]^{\sim}_{\phi}
\ar@{->>}[dd]^{\textrm{the natural surjection}}
 & \mathrm{Gal}(K_{\mathcal{O},\,N}/K) \ar@{->>}[dd]^{\textrm{the restriction homomorphism}} 
\ar@{->}[r]^{\sim}_{\rho\circ\phi^{-1}} & I(\mathcal{O},\,N)/P_{1,\,N}(\mathcal{O},\,N) \ar@{->>}[dd]^{\textrm{the natural homomorphism}}
\\\\ 
 \mathcal{C}_\Gamma(D_\mathcal{O},\,N) \ar@{->}[r]^{\textrm{bijective}}_{\psi}
 & \mathrm{Gal}(L/K) \ar@{->}[r]^{\sim} & I(\mathcal{O},\,N)/P
 }
\end{equation*}
\caption{A commutative diagram of class groups and Galois groups}
\label{Figure4}
\end{figure}

\end{proof}

For a subgroup $H$ of $\mathrm{SL}_2(\mathbb{Z})$ containing $\Gamma_1(N)$, let 
$W_{D_\mathcal{O},\,N,\,H}$ be the subgroup of $\mathrm{GL}_2(\widehat{\mathbb{Z}})$ defined by
\begin{equation*}
W_{D_\mathcal{O},\,N,\,H}=\left\langle
(\gamma_p)_p\in\prod_p\mathrm{GL}_2(\mathbb{Z}_p)~\Bigg|~
\begin{array}{l}\textrm{there are $ax^2+bxy+cy^2\in\mathcal{Q}(D_\mathcal{O},\,N)$ and $\gamma\in H$}\\
\textrm{so that}~\gamma_p\equiv\begin{bmatrix}1&0\\0&a\end{bmatrix}\gamma\Mod{NM_2(\mathbb{Z}_p)}~
\textrm{for all primes $p$}
\end{array}
\right\rangle.
\end{equation*}

\begin{corollary}
Let $N$ be a positive integer and $H$ be a subgroup of $\mathrm{SL}_2(\mathbb{Z})$ containing $\Gamma_1(N)$. 
Then, the congruence subgroup
$\mathbb{Q}^\times W_{D_\mathcal{O},\,N,\,H}\cap\mathrm{SL}_2(\mathbb{Z})$
induces a form class group of discriminant $D_\mathcal{O}$
and level $N$. 
\end{corollary}
\begin{proof}
Set
$S=\mathbb{Q}^\times W_{D_\mathcal{O},\,N,\,H}$ which 
is an open subgroup of $\mathrm{GL}_2(\widehat{\mathbb{Q}})$ containing $\mathbb{Q}^\times$ such that
$S/\mathbb{Q}^\times$ is compact. If we let $\Gamma=\mathbb{Q}^\times W_{D_\mathcal{O},\,N,\,H}\cap\mathrm{SL}_2(\mathbb{Z})$, then we see that
\begin{equation*}
\Gamma_S/\mathbb{Q}^\times\simeq\Gamma/\langle-I_2\rangle,\quad
k_S=k_{D_\mathcal{O},\,N}\quad\textrm{and}\quad S\supseteq\Delta_S. 
\end{equation*}
Therefore we conclude by Theorem \ref{S} that 
$\Gamma$ induces a form class group of discriminant $D_\mathcal{O}$ and level $N$. 
\end{proof}

\section {Canonical models associated with form class groups}

In this section, we shall show that
under certain assumptions
the converse of Theorem \ref{S} does hold. 
\par
Each congruence subgroup $\Gamma$ of $\mathrm{SL}_2(\mathbb{Z})$ acts on 
the set $\mathcal{Q}(D_\mathcal{O},\,1)$ in a natural way\,:
\begin{equation*}
Q^\gamma=Q\left(\gamma\begin{bmatrix}x\\y\end{bmatrix}\right)\quad(Q\in\mathcal{Q}(D_\mathcal{O},\,N),\,
\gamma\in\Gamma). 
\end{equation*}
However, it does not necessarily act on $\mathcal{Q}(D_\mathcal{O},\,N)$, that is, 
$Q^\gamma$ may not belong to $\mathcal{Q}(D_\mathcal{O},\,N)$ for some $Q\in\mathcal{Q}(D_\mathcal{O},\,N)$
and $\gamma\in\Gamma$. Note that 
when $\Gamma$ acts on $\mathcal{Q}(D_\mathcal{O},\,N)$, the equivalence relation 
$\sim_\Gamma$ on $\mathcal{Q}(D_\mathcal{O},\,N)$ is just the one induced from this action. 
For $\gamma=\begin{bmatrix}q&r\\s&t\end{bmatrix}\in\mathrm{SL}_2(\mathbb{Z})$ and $\tau\in\mathbb{H}$, we let
$j(\gamma,\tau)=s\tau+t$. 

\begin{lemma}\label{set}
Let $\Gamma$ be a subgroup of $\mathrm{SL}_2(\mathbb{Z})$ containing 
$\langle\Gamma_1(N),\,-I_2\rangle$.  Assume that $\Gamma$ 
acts on $\mathcal{Q}(D_\mathcal{O},\,N)$ and
induces a form class group of discriminant $D_\mathcal{O}$ and level $N$.
We further assume that $D_\mathcal{O}$ is different from $-3$ and $-4$. 
\begin{enumerate}
\item[\textup{(i)}] 
If $\alpha\in\mathrm{SL}_2(\mathbb{Z})$ satisfies 
\begin{equation}\label{**}
\alpha\equiv\begin{bmatrix}\mathrm{*}&\mathrm{*}\\
as&t\end{bmatrix}\Mod{NM_2(\mathbb{Z})}
\end{equation}
for some $ax^2+bxy+cy^2\in\mathcal{Q}(D_\mathcal{O},\,N)$ and 
$\begin{bmatrix}q&r\\s&t\end{bmatrix}\in\Gamma$,
then $\alpha$ belongs to $\Gamma$. 
\item[\textup{(ii)}] The group $W_{D_\mathcal{O},\,N,\,\Gamma}$ coincides with the set
\begin{equation*}
\widetilde{W}_{D_\mathcal{O},\,N,\,\Gamma}=\left\{
(\gamma_p)_p\in\prod_p\mathrm{GL}_2(\mathbb{Z}_p)~\Bigg|~
\begin{array}{l}\textrm{there are $ax^2+bxy+cy^2\in\mathcal{Q}(D_\mathcal{O},\,N)$ and $\gamma\in\Gamma$}\\
\textrm{so that}~\gamma_p\equiv\begin{bmatrix}1&0\\0&a\end{bmatrix}\gamma\Mod{NM_2(\mathbb{Z}_p)}~
\textrm{for all primes $p$}
\end{array}
\right\}. 
\end{equation*}
\end{enumerate}
\end{lemma}
\begin{proof}
\begin{enumerate}
\item[(i)] Since $\Gamma$ induces a form class group of discriminant $D_\mathcal{O}$ and level $N$,
there is a subgroup $P$ of $P(\mathcal{O},\,N)$ containing $P_{1,\,N}(\mathcal{O},\,N)$ so that
the mapping
\begin{equation*}
\begin{array}{rccc}
\phi:&\mathcal{C}_\Gamma(D_\mathcal{O},\,N)&\rightarrow& I(\mathcal{O},\,N)/P,\\
&[[Q]]&\mapsto&[[\omega_Q,\,1]]
\end{array}
\end{equation*}
is a well-defined bijection. 
Now, let $\alpha$ be an element of $\mathrm{SL}_2(\mathbb{Z})$ satisfying
(\ref{**}) with $Q=ax^2+bxy+cy^2\in\mathcal{Q}(D_\mathcal{O},\,N)$ and
$\gamma=\begin{bmatrix}q&r\\s&t\end{bmatrix}\in\Gamma$. 
Then we see that in $I(\mathcal{O},\,N)/P$
\begin{equation}\label{wppw}
[[\omega_Q,\,1]]=\phi([Q])=\phi([Q^{\gamma^{-1}}])=[[\omega_{Q^{\gamma^{-1}}},\,1]]=[[\gamma(\omega_Q),\,1]]=
\left[\frac{1}{j(\gamma,\,\omega_Q)}\mathcal{O}[\omega_Q,\,1]\right].
\end{equation}
On the other hand, we find that in $I(\mathcal{O},\,N)/P$
\begin{eqnarray*}
[j(\gamma,\,\omega_Q)\mathcal{O}]&=&[(s\omega_Q+t)\mathcal{O}]\\
&=&[((as)\omega_{Q'}+t)\mathcal{O}]
\quad\textrm{where}~Q'=x^2+bxy+\frac{b^2-D_\mathcal{O}}{4}y^2~(\in\mathcal{Q}(D_\mathcal{O},\,N))\\
&=&[j(\alpha,\,\omega_{Q'})\mathcal{O}]\quad\textrm{since
$\omega_{Q'}\in\mathcal{O}$, 
$(as)\omega_{Q'}+t\equiv j(\alpha,\,\omega_{Q'})\Mod{N\mathcal{O}}$}\\
&&\hspace{2.5cm}\textrm{and $P$ contains $P_{1,\,N}(\mathcal{O})$},
\end{eqnarray*}
which yields by (\ref{wppw}) that $j(\alpha,\,\omega_{Q'})\mathcal{O}\in P$. 
And, we derive that
\begin{equation*}
\phi([Q'])=[[\omega_{Q'},\,1]]
=\left[
\frac{1}{j(\alpha,\,\omega_{Q'})}\mathcal{O}[\omega_{Q'},\,1]\right]
=[[\alpha(\omega_{Q'}),\,1]]=[[\omega_{Q'^{\alpha^{-1}}},\,1]]=\phi([Q'^{\alpha^{-1}}]),
\end{equation*}
where $Q'^{\alpha^{-1}}$ is understood as the action of $\alpha^{-1}$
on the element $Q'$ of $\mathcal{Q}(D_\mathcal{O},\,1)$. 
Thus it follows that $[Q'^{\alpha^{-1}}]=[Q']$ in $\mathcal{C}_\Gamma(D_\mathcal{O},\,N)$ because $\phi$ is injective,
and hence $Q'^{\alpha^{-1}}=Q'^\beta$ for some $\beta\in\Gamma$. 
So, we attain that $\beta\alpha$ belongs to the isotropy subgroup of $Q'$ in $\mathrm{SL}_2(\mathbb{Z})$, namely, 
$\{I_2,\,-I_2\}$ because $D_\mathcal{O}\neq-3,\,-4$ (cf. \cite[Proposition 1.5 (c)]{Silverman94}).
Therefore $\alpha$ belongs to $\Gamma$. 
\item[(ii)]
The inclusion $\widetilde{W}_{D_\mathcal{O},\,N,\,\Gamma}\subseteq
W_{D_\mathcal{O},\,N,\,\Gamma}$ is obvious. 
\par
For the converse inclusion, it suffices to show that
$\widetilde{W}_{D_\mathcal{O},\,N,\,\Gamma}$ is a subgroup of 
$\mathrm{GL}_2(\widehat{\mathbb{Z}})$. Observe
first that the identity element of $\mathrm{GL}_2(\widehat{\mathbb{Z}})$ belongs to 
$\widetilde{W}_{D_\mathcal{O},\,N,\,\Gamma}$ since
$Q_\mathcal{O}=x^2+b_\mathcal{O}xy+c_\mathcal{O}y^2\in\mathcal{Q}(D_\mathcal{O},\,N)$ 
and $I_2\in\Gamma$. 
Now, let $(\alpha_p)_p,\,(\beta_p)_p\in \widetilde{W}_{D_\mathcal{O},\,N,\,\Gamma}$. 
Then we get that for every prime $p$
\begin{equation*}
\alpha_p\equiv\begin{bmatrix}1&0\\0&a\end{bmatrix}\gamma\Mod{NM_2(\mathbb{Z}_p)}
\quad\textrm{and}\quad
\beta_p\equiv\begin{bmatrix}1&0\\0&a'\end{bmatrix}\gamma'\Mod{NM_2(\mathbb{Z}_p)}
\end{equation*}
for some $Q=ax^2+bxy+cy^2,\,Q'=a'x^2+b'xy+c'y^2\in\mathcal{Q}(D_\mathcal{O},\,N)$ and $\gamma,\,\gamma'\in\Gamma$. 
And, we achieve that
\begin{eqnarray*}
\beta_p\alpha_p^{-1}&\equiv&\begin{bmatrix}
1&0\\0&a'\end{bmatrix}\gamma''\begin{bmatrix}1&0\\0&\check{a}\end{bmatrix}
\Mod{NM_2(\mathbb{Z}_p)}\quad\textrm{where}~\gamma''
=\begin{bmatrix}q&r\\s&t\end{bmatrix}=\gamma'\gamma^{-1}~(\in\Gamma)\\
&&\hspace{4.5cm}\textrm{and $\check{a}$ is any integer such that $a\check{a}\equiv1\Mod{N}$}\\
&\equiv&\begin{bmatrix}1&0\\0&a'\check{a}\end{bmatrix}\begin{bmatrix}q&\check{a}r\\as&t
\end{bmatrix}\Mod{NM_2(\mathbb{Z}_p)}\\
&\equiv&\begin{bmatrix}1&0\\0&a''\end{bmatrix}\gamma'''\Mod{NM_2(\mathbb{Z}_p)}\\
&&\textrm{for some $a''x^2+b''xy+c''y^2\in\mathcal{Q}(D_\mathcal{O},\,N)$ and $\gamma'''\in\Gamma$
by Remark \ref{cyclotomic}, (i) and}\\
&&\textrm{the surjectivity of the reduction $\mathrm{SL}_2(\mathbb{Z})\rightarrow
\mathrm{SL}_2(\mathbb{Z}/N\mathbb{Z})$ (cf. \cite[Lemma 1.38]{Shimura})}.
\end{eqnarray*}
This observation implies that $(\beta_p)_p^{-1}(\alpha_p)_p\in \widetilde{W}_{D_\mathcal{O},\,N,\,\Gamma}$,
and hence $\widetilde{W}_{D_\mathcal{O},\,N,\,\Gamma}$ is a subgroup of $\mathrm{GL}_2(\widehat{\mathbb{Z}})$.
\end{enumerate}
\end{proof}

\begin{theorem}\label{main}
Let $\Gamma $ be a subgroup of $\mathrm{SL}_2(\mathbb{Z})$ containing $\Gamma_1(N)$.
Let $K$ be an imaginary quadratic field and $\mathcal{O}$ be an order of discriminant $D_\mathcal{O}$ in $K$. 
Assume that
\begin{enumerate}
\item[\textup{(a)}] $\Gamma$ acts on $\mathcal{Q}(D_\mathcal{O},\,N)$,
\item[\textup{(b)}] $D_\mathcal{O}$ is different from $-3$ and $-4$.
\end{enumerate}
Then, $\Gamma$ induces a form class group of discriminant $D_\mathcal{O}$ and level $N$ if and only if 
there is an open subgroup $S$ of $\mathrm{GL}_2(\widehat{\mathbb{Q}})$ containing $\mathbb{Q}^\times$ such that
$S/\mathbb{Q}^\times$ is compact and
\begin{enumerate}
\item[\textup{(i)}] $\Gamma_S/\mathbb{Q}^\times\simeq\langle\Gamma,\,-I_2\rangle/\langle-I_2\rangle$,
\item[\textup{(ii)}] $k_S$ is a subfield of $k_{D_\mathcal{O},\,N}=\mathbb{Q}(\zeta_N)\cap K$,
\item[\textup{(iii)}] $S$ contains $\Delta_S$.
\end{enumerate}
\end{theorem}
\begin{proof}
The ``\,if\,'' part of the theorem, even without the hypotheses (a) and (b), is given in Theorem \ref{S}. 
\par
Conversely, assume that $\Gamma$ induces a form class group of discriminant $D_\mathcal{O}$ and level $N$. 
Since $\langle\Gamma,\,-I_2\rangle$ ($\subseteq\mathrm{SL}_2(\mathbb{Z})$) induces the same form class group as that induced from $\Gamma$, 
we may assume that $\Gamma$ contains $-I_2$. 
Set 
\begin{equation*}
S=\mathbb{Q}^\times W_{D_\mathcal{O},\,N,\,\Gamma}
\end{equation*}
which is an open subgroup of $\mathrm{GL}_2(\widehat{\mathbb{Q}})$ containing $\mathbb{Q}^\times$ such that
$S/\mathbb{Q}^\times$ is compact. 
Then we get from the definition of $W_{D_\mathcal{O},\,N,\,\Gamma}$ that $S$ satisfies (ii) and (iii). 
And, we attain by (a), (b) and Lemma \ref{set} that 
\begin{equation*}
\Gamma_S=S\cap\mathrm{GL}_2^+(\mathbb{Q})=\mathbb{Q}^\times\Gamma,
\end{equation*}
and so $S$ also satisfies (i). 
This completes the proof. 
\end{proof}

\section {Congruence subgroups acting on $\mathcal{Q}(D_\mathcal{O},\,N)$}

In this last section, we shall find an equivalent condition for (a) of Theorem \ref{main}. 
\par For a positive integer $n$, let $\Gamma_0(n)$ be the Hecke congruence subgroup of $\mathrm{SL}_2(\mathbb{Z})$ of level $n$,
that is,
\begin{equation*}
\Gamma_0(n)=\left\{
\gamma\in\mathrm{SL}_2(\mathbb{Z})~|~\gamma\equiv\begin{bmatrix}\mathrm{*}&\mathrm{*}\\
0&\mathrm{*}\end{bmatrix}\Mod{nM_2(\mathbb{Z})}\right\}.
\end{equation*}

\begin{lemma}\label{pr}
Let $\Gamma$ be a subgroup of $\mathrm{SL}_2(\mathbb{Z})$ containing 
$\Gamma_1(N)$, and let $M$ \textup{(}$\geq2$\textup{)} be a square-free positive integer. 
If $\Gamma$ is not contained in $\Gamma_0(M)$, then 
there is a prime factor $p$ of $M$ and $\begin{bmatrix}q&r\\s&t\end{bmatrix}\in\Gamma$ such that
$p\,|\,q$ and $p\,\nmid\,s$.
\end{lemma}
\begin{proof}
Since $\Gamma\not\subseteq\Gamma_0(M)$, we can 
take a matrix $\gamma=\begin{bmatrix}q_0&r_0\\s_0&t_0\end{bmatrix}$ in $\Gamma$ such that $s_0\not\equiv0
\Mod{M}$.  Moreover, since $M$ is square-free, there is  a prime factor $p$ of $M$ not dividing $s_0$. 
Then the linear congruence $s_0x\equiv-q_0\Mod{p}$ has an integer solution, say $x=k$. 
Now that $\Gamma$ contains $T=\begin{bmatrix}1&1\\0&1\end{bmatrix}$, it also has the matrix
\begin{equation*}
\begin{bmatrix}q&r\\s&t\end{bmatrix}=T^k\gamma=
\begin{bmatrix}1&k\\0&1\end{bmatrix}\begin{bmatrix}q_0&r_0\\s_0&t_0\end{bmatrix}
=\begin{bmatrix}q_0+ks_0&r_0+kt_0\\s_0&t_0\end{bmatrix}\quad
\textrm{satisfying}~p\,|\,q~\textrm{and}~p\,\nmid\,s. 
\end{equation*}
\end{proof}

\begin{lemma}\label{x^2}
If $Q=ax^2+bxy+cy^2\in\mathcal{Q}(D_\mathcal{O},\,N)$ and $\gamma=\begin{bmatrix}q&r\\s&t\end{bmatrix}\in\Gamma$, then
the coefficient of $x^2$ in $Q\left(\gamma\begin{bmatrix}x\\y\end{bmatrix}\right)$ is
$Q(q,\,s)=aq^2+bqs+cs^2$. 
\end{lemma}
\begin{proof}
Straightforward. 
\end{proof}

By $M_{D_\mathcal{O},\,N}$ we mean the product of all prime factors $p$ of $N$ satisfying $\left(\frac{D_\mathcal{O}}{p}\right) \neq-1$, where $\left(\frac{\cdot}{p}\right)$ denotes the Kronecker symbol.

\begin{theorem}\label{actioncondition}
Let $\Gamma$ be a subgroup of $\mathrm{SL}_2(\mathbb{Z})$ containing $\Gamma_1(N)$. 
Then, $\Gamma$ acts on $\mathcal{Q}(D_\mathcal{O},\,N)$ if and only if
$\Gamma$ is contained in $\Gamma_0(M_{D_\mathcal{O},\,N})$. 
\end{theorem}
\begin{proof}
Assume that $\Gamma$ acts on $\mathcal{Q}(D_\mathcal{O},\,N)$. Suppose on the contrary that
$\Gamma$ is not contained in $\Gamma_0(M_{D_\mathcal{O},\,N})$. 
By Lemma \ref{pr}, there is a 
prime factor $p$ of $M_{D_\mathcal{O},\,N}$ and 
$\gamma=\begin{bmatrix}q&r\\s&t\end{bmatrix}\in\Gamma$ 
such that 
\begin{equation}\label{tptr}
p\,|\,q\quad\textrm{and}\quad p\,\nmid\,s. 
\end{equation}
There are three possible cases\,:
$\left\{\begin{array}{l}
\textrm{$p$ divides $D_\mathcal{O}$, or}\\
\textrm{$p$ is odd and $\displaystyle\left(\frac{D_\mathcal{O}}{p}\right)=1$, or}\\ 
\textrm{$p=2$ and $D_\mathcal{O}\equiv1\Mod{8}$}.
\end{array}\right.$
\begin{enumerate}
\item[Case 1.] First, consider the case where $p$ divides $D_\mathcal{O}$. 
\begin{enumerate}
\item[Case 1-1.] If $p$ is odd and $D_\mathcal{O}\equiv1\Mod{4}$, then we have
$p^2\equiv D_\mathcal{O}\Mod{4p}$, and hence
\begin{equation}\label{tD0t}
\frac{p^2-D_\mathcal{O}}{4}\equiv0\Mod{p}. 
\end{equation}
We then see that for $\displaystyle
Q=x^2+pxy+\frac{p^2-D_\mathcal{O}}{4}y^2$ ($\in\mathcal{Q}(D_\mathcal{O},\,N)$)
\begin{eqnarray*}
\textrm{the coefficient of $x^2$ in $Q^\gamma$}&=&q^2+pqs+\frac{p^2-D_\mathcal{O}}{4}s^2\quad\textrm{by 
Lemma \ref{x^2}}\\
&\equiv&0\Mod{p}\quad\textrm{by (\ref{tptr}) and (\ref{tD0t})}.
\end{eqnarray*}
\item[Case 1-2.] If $p$ is odd and $D_\mathcal{O}\equiv0\Mod{4}$, then we get that
$4p^2\equiv D_\mathcal{O}\Mod{4p}$, and so
\begin{equation}\label{4tD0t}
\frac{4p^2-D_\mathcal{O}}{4}\equiv0\Mod{p}.
\end{equation}
Observe that for $\displaystyle
Q=x^2+2pxy+\frac{4p^2-D_\mathcal{O}}{4}y^2$ ($\in\mathcal{Q}(D_\mathcal{O},\,N)$)
\begin{eqnarray*}
\textrm{the coefficient of $x^2$ in $Q^\gamma$}&=&q^2+2pqs+\frac{4p^2-D_\mathcal{O}}{4}s^2\quad\textrm{by 
Lemma \ref{x^2}}\\
&\equiv&0\Mod{p}\quad\textrm{by (\ref{tptr}) and (\ref{4tD0t})}.
\end{eqnarray*}
\item[Case 1-3.] If $p=2$ and $D_\mathcal{O}\equiv0\Mod{8}$, then 
$\displaystyle Q=x^2-\frac{D_\mathcal{O}}{4}y^2$ belongs to $\mathcal{Q}(D_\mathcal{O},\,N)$ and
\begin{eqnarray*}
\textrm{the coefficient of $x^2$ in $Q^\gamma$}&=&q^2-\frac{D_\mathcal{O}}{4}s^2\quad\textrm{by 
Lemma \ref{x^2}}\\
&\equiv&0\Mod{p}\quad\textrm{by (\ref{tptr}) and the fact $\frac{D_\mathcal{O}}{4}\equiv0\Mod{2}$}.
\end{eqnarray*}
\item[Case 1-4.] If $p=2$ and $D_\mathcal{O}\equiv4\Mod{8}$,
then
$\displaystyle Q=x^2+2xy+\frac{4-D_\mathcal{O}}{4}y^2$ belongs to $\mathcal{Q}(D_\mathcal{O},\,N)$ and
\begin{eqnarray*}
\textrm{the coefficient of $x^2$ in $Q^\gamma$}&=&q^2+2qs+\frac{4-D_\mathcal{O}}{4}s^2\quad\textrm{by 
Lemma \ref{x^2}}\\
&\equiv&0\Mod{p}\quad\textrm{by (\ref{tptr}) and the fact $\frac{4-D_\mathcal{O}}{4}\equiv0\Mod{2}$}.
\end{eqnarray*}
\end{enumerate}
\item[Case 2.] Second, consider the case where $p$ is odd and $\displaystyle\left(\frac{D_\mathcal{O}}{p}\right)=1$.
Then, the quadratic congruence
$x^2\equiv D_\mathcal{O}\Mod{4p}$ has an integer solution, say $x=b_0$.
We see that
\begin{equation}\label{b0}
\frac{b_0^2-D_\mathcal{O}}{4}\equiv0\Mod{p}.
\end{equation}
And, we find that for $\displaystyle Q=x^2+b_0xy+\frac{b_0^2-D_\mathcal{O}}{4}y^2$ ($\in\mathcal{Q}(D_\mathcal{O},\,N)$)
\begin{eqnarray*}
\textrm{the coefficient of $x^2$ in $Q^\gamma$}&=&q^2+b_0qs+\frac{b_0^2-D_\mathcal{O}}{4}s^2\quad\textrm{by 
Lemma \ref{x^2}}\\
&\equiv&0\Mod{p}\quad\textrm{by (\ref{tptr}) and (\ref{b0})}.
\end{eqnarray*}
\item[Case 3.] Third, consider the case where $p=2$ and $D_\mathcal{O}\equiv1\Mod{8}$. Then we claim that
$\displaystyle Q=x^2+xy+\frac{1-D_\mathcal{O}}{4}y^2$ belongs to $\mathcal{Q}(D_\mathcal{O},\,N)$ and
\begin{eqnarray*}
\textrm{the coefficient of $x^2$ in $Q^\gamma$}&=&q^2+qs+\frac{1-D_\mathcal{O}}{4}s^2\quad\textrm{by 
Lemma \ref{x^2}}\\
&\equiv&0\Mod{p}\quad\textrm{by (\ref{tptr}) and the fact $\frac{1-D_\mathcal{O}}{4}\equiv0\Mod{2}$}.
\end{eqnarray*}
\end{enumerate}
Each case contradicts the assumption that $\Gamma$ acts on $\mathcal{Q}(D_\mathcal{O},\,N)$. 
Therefore we conclude that $\Gamma$ must be contained in $\Gamma_0(M_{D_\mathcal{O},\,N})$. 
\par
Conversely, assume that $\Gamma$ is contained in $\Gamma_0(M_{D_\mathcal{O},\,N})$. If $N=1$, then the assertion is obvious. 
Now, we let $N\geq2$. 
Let $Q=ax^2+bxy+cy^2\in\mathcal{Q}(D_\mathcal{O},\,N)$ and $\gamma=\begin{bmatrix}q&r\\s&t\end{bmatrix}\in\Gamma$.
Let $a'$ be the coefficient of $x^2$ in $Q\left(\gamma\begin{bmatrix}x\\y\end{bmatrix}\right)$. 
For each prime factor $p$ of $N$, we have either
$p\,|\,s$ or $p\,\nmid\,s$. 
\begin{enumerate}
\item[Case (i)] If $p\,|\,s$, then we derive that
\begin{eqnarray*}
\gcd(a',\,p)
&=&\gcd(aq^2+bqs+cs^2,\,p)
\quad\textrm{by Lemma \ref{x^2}}\\
&=&\gcd(aq^2,\,p)\\
&=&1\quad\textrm{because}~\gcd(a,\,N)=1~\textrm{and}~\det(\gamma)=1~\textrm{implies}~p\,\nmid\,q.
\end{eqnarray*}
\item[Case (ii)] If $p\,\nmid\,s$, then $p\,\nmid\,M_{D_\mathcal{O},\,N}$ since $\gamma\in\Gamma_0(M_{D_\mathcal{O},\,N})$.  
Thus we have
\begin{equation}\label{Dt}
\left(\frac{D_\mathcal{O}}{p}\right)=-1.
\end{equation}
\begin{enumerate}
\item[Case (ii)-1.] Consider the case where $p$ is odd.
The discriminant of the quadratic congruence
$ax^2+bsx+cs^2\equiv0\Mod{p}$ with $a,\,s\not\equiv0\Mod{p}$
is 
\begin{equation*}
(bs)^2-4a(cs^2)=s^2(b^2-4ac)=s^2D_\mathcal{O}. 
\end{equation*}
Since $\displaystyle\left(\frac{s^2D_\mathcal{O}}{p}\right)=
\left(\frac{D_\mathcal{O}}{p}\right)=-1$ by (\ref{Dt}), the quadratic congruence has no integer solution. 
It then follows from Lemma \ref{x^2} that
\begin{equation*}
\gcd(a',\,p)=\gcd(aq^2+bqs+cs^2,\,p)=1. 
\end{equation*}
\item[Case (ii)-2.] If $p=2$, then we get by (\ref{Dt}) that $D_\mathcal{O}\equiv5\Mod8$ and so $2\,\nmid\,c$.
Thus we deduce that
\begin{eqnarray*}
a'&=&aq^2+bqs+cs^2\quad\textrm{by 
Lemma \ref{x^2}}\\
&\equiv&q^2+bq+1\Mod{2}\quad\textrm{because}~2\,\nmid\,a, ~2\,\nmid\,s,~2\,\nmid\,c\\
&\equiv&q^2+q+1\Mod{2}\quad\textrm{since $b$ is odd by the fact $b^2-4ac=D_\mathcal{O}\equiv5\Mod{8}$}\\
&\equiv&1\Mod{2},
\end{eqnarray*}
and hence $\gcd(a',\,p)=1$. 
\end{enumerate}
\end{enumerate}
This observation yields that $\Gamma$ acts on $\mathcal{Q}(D_\mathcal{O},\,N)$. 
\end{proof}

\bibliographystyle{amsplain}

\address{
Department of Mathematical Sciences \\
KAIST \\
Daejeon 34141\\
Republic of Korea} {jkgoo@kaist.ac.kr}
\address{
Department of Mathematics\\
Hankuk University of Foreign Studies\\
Yongin-si, Gyeonggi-do 17035\\
Republic of Korea} {dhshin@hufs.ac.kr}
\address{
Department of Mathematics Education\\
Pusan National University\\
Busan 46241\\Republic of Korea}
{dsyoon@pusan.ac.kr}
\end{document}